\documentclass[12pt]{article}

\usepackage[l2tabu,orthodox]{nag}
\usepackage[all,error]{onlyamsmath}

\usepackage{amsmath}
\usepackage{amsfonts}
\usepackage{amsthm}
\usepackage{amssymb}
\usepackage{dsfont}
\usepackage[shortlabels]{enumitem}
\usepackage{graphicx}
\usepackage{mathrsfs}
\usepackage[normalem]{ulem}
\usepackage[utf8]{inputenc}
\usepackage[hidelinks,colorlinks=true,urlcolor=blue,linkcolor=black,citecolor=black]{hyperref}

\allowdisplaybreaks

\newcommand*{\doi}[1]{\href{http://dx.doi.org/\detokenize{#1}}{doi}}
\setlist{topsep=0em, partopsep=0pt, itemsep=0pt}

\newtheorem{theorem}[equation]{Theorem}
\newtheorem{proposition}[equation]{Proposition}
\newtheorem{lemma}[equation]{Lemma}

\theoremstyle{remark}

\newcommand{\1}{\mathds{1}}
\newcommand{\E}{\mathbb{E}}
\newcommand{\bE}{\mathbb{E}}
\newcommand{\bP}{\mathbb{P}}
\newcommand{\F}{\mathscr{F}}
\newcommand{\N}{\mathbb{N}}
\newcommand{\Pb}{\mathbb{P}}
\newcommand{\Z}{\mathbb{Z}}
\newcommand{\I}{\mathds{1}}
\newcommand{\0}{{\boldsymbol{0}}}

\newcommand{\NN}{\mathcal{N}}
\newcommand{\dd}{{\mathrm d}}

\newcommand{\tm}{\boldsymbol{m}}
\newcommand{\T}{\mathcal{T}}

\newcommand{\sleep}{\mathfrak{s}}
\newcommand{\toppling}{\mathfrak{t}}

\renewcommand{\leq}{\leqslant}
\renewcommand{\le}{\leqslant}
\renewcommand{\geq}{\geqslant}
\renewcommand{\ge}{\geqslant}
\renewcommand{\epsilon}{\varepsilon}

\begin{document}

\title{Diffusive bounds for the critical \\ density of activated random walks}

\author{Amine Asselah
\and
Leonardo T. Rolla
\and
Bruno Schapira
}
\maketitle

\begin{abstract}
We consider symmetric activated random walks on $\mathbb{Z}$, and show that the critical density $\zeta_c$ satisfies $c\sqrt{\lambda} \leq \zeta_c(\lambda) \leq C \sqrt{\lambda}$ where $\lambda$ denotes the sleep rate.
\end{abstract}

\section{Introduction}
\label{sec-intro}

The Activated Random Walk model is a system of interacting random walks that we consider on the graph $\Z$.
Each walk performs a continuous-time simple symmetric random walk, and falls asleep at an exponential time of parameter $\lambda$.
When it falls asleep, the walk stays still.
When not sleeping, we call it active.
When an active walk meets a sleeping walk, the latter is reactivated and resumes its movement.

An important property of this model is that it has an \emph{absorbing-state phase transition}.
With an initial density of walks below a critical value, $\zeta_c(\lambda)$, the system fixates, that is all walks eventually sleep.
Above $\zeta_c$, the system stays active, that is, each walk is reactivated
infinitely many times.

This model was popularized in~\cite{DickmanRollaSidoravicius10}, and several non-trivial bounds for $\zeta_c=\zeta_c(\lambda)$ on the graph $\Z^d$ were proved in the past few years.
For $d=1$, it was proved in~\cite{RollaSidoravicius12} that $\zeta_c>0$ for all $\lambda$ and $\zeta_c \to 1$ as $\lambda\to\infty$.
For $d\ge 2$ and $\lambda=\infty$, it was also shown in~\cite{Shellef10} that $\zeta_c>0$ and in~\cite{CabezasRollaSidoravicius14,CabezasRollaSidoravicius18} that $\zeta_c \ge 1$.
For $d \ge 2$ and $\lambda>0$ it was shown in~\cite{SidoraviciusTeixeira17} that $\zeta_c>0$, assuming short-range unbiased jump distributions.
This was extended to general jump distributions in~\cite{StaufferTaggi18}, where it was also shown that $\zeta_c \to 1$ as $\lambda\to\infty$.
It was proved in~\cite{AmirGurel-Gurevich10,Shellef10} that $\zeta_c \le 1$ in any dimension for any $\lambda$.
For biased jump distributions, it was shown in~\cite{Taggi16} that, on $d=1$, $\zeta_c<1$ for every $\lambda$ and $\zeta_c\to 0$ as $\lambda \to 0$, and on $d \ge 2$ that $\zeta_c<1$ for small $\lambda$.
The picture on $d \geq 2$ was extended in~\cite{RollaTournier18} by showing that $\zeta_c<1$ for every $\lambda$ and $\zeta_c \to 0$ as $\lambda\to 0$.
For unbiased jumps, it was shown that $\zeta_c\to 0$ as $\lambda \to 0$, in~\cite{BasuGangulyHoffman18} for $d=1$ and~\cite{StaufferTaggi18} for $d \ge 3$, and finally~\cite{Taggi19} that $\zeta_c<1$ for every $\lambda<\infty$ and $d \ge 3$.
Proving that $\zeta_c < 1$ for some $\lambda$ in $d=2$ is still open.
In~\cite{RollaSidoraviciusZindy19} it was shown that this phase transition is determined by the density alone and is independent of particular details of the initial state.
Existence of slow stabilization phase and a fast stabilization phase for a conservative finite-volume dynamics was studied in~\cite{BasuGangulyHoffmanRichey19}.

Our main result gives a diffusive upper bound for $\zeta_c(\lambda)$ when $d=1$, improving the recent result by Basu, Ganguly and Hoffman~\cite{BasuGangulyHoffman18}.

\begin{theorem}\label{thm:main}
There are positive constants $c$ and $C$, such that for all $\lambda>0$,
\begin{equation}
\label{main-estimate}
c\sqrt{\lambda} \leq \zeta_c(\lambda) \leq C \sqrt{\lambda}.
\end{equation}
\end{theorem}

The lower bound follows from the procedure introduced in~\cite{RollaSidoravicius12}, see~§\ref{sec:lower}.

As in~\cite{BasuGangulyHoffman18}, in order to get the upper bound in~\eqref{main-estimate} we prove a quantitative estimate for the \emph{finite-volume dynamics}, defined as follows.
For a finite interval $V\subseteq \Z$, consider
a dynamics such that walks are lost forever when they escape $V$.
This process eventually stabilizes,
when all walks left in $V$ are sleeping, and
its law is denoted $\bP^{\eta_0}_V$ when
the initial configuration is $\eta_0$.

To prove activity the following condition is enough.
For a finite domain $V$, we call $S(V)$ the number of sleeping walks in $V$ after stabilization of $V$ in the finite-domain dynamics.
For $r\ge 1$, we let $V_r:=\{-r,\dots,r\}$.

\begin{theorem}\label{thm:quantitative}
There exist positive constants $\alpha$ and $\beta$,
such that for any configuration $\eta_0$, any $\lambda>0$ and any $r\ge 1$,
\begin{equation}
\label{eq:expmoment}
\bE^{\eta_0}_{V_r} [ e^{\alpha S(V_r)} ] \le e^{\beta \sqrt \lambda\cdot r}.
\end{equation}
\end{theorem}

Theorem~\ref{thm:quantitative} implies that
there are positive constants $C$ and $c$, such that for any $r$ integer, and
any initial configuration $\eta_0$,
\begin{equation}\label{eq:conditionu}
\bP^{\eta_0}_{V_r}\big(S(V_r)\ge 2 C \sqrt{\lambda} r\big) \le e^{-c r}.
\end{equation}
As discussed in §\ref{sec:definitions} below, this in turn implies
that every $\zeta > C \sqrt{\lambda}$ is in the active phase of the ARW model,
which gives the upper bound of Theorem~\ref{thm:main} with the same constant $C$.

Another consequence of Theorem~\ref{thm:quantitative} is that every $\zeta > C \sqrt{\lambda}$ is in the ``metastable'' phase for the fixed-energy version of the ARW, in the following sense.
Consider the ARW dynamics on the ring $\Z_n = \Z / n \Z$, with initial condition i.i.d.\ Poisson of mean $\zeta > C \sqrt{\lambda}$.
Let $\mathcal{T}$ denote the total activity in the system, measured by adding the total time each site is occupied by active walks,
counted with multiplicity.

\begin{theorem}
\label{thm:slow}
For some $c$ depending on $\lambda$ and $\zeta>C \sqrt{\lambda}$, for all large $n$,
\begin{equation}
\bP_{\Z_n}\big( \T \ge e^{cn} \big) \ge 1 - e^{-c n}.
\end{equation}
\end{theorem}

Theorem~\ref{thm:slow} follows from Theorem~\ref{thm:quantitative}, see~\cite[§4]{BasuGangulyHoffmanRichey19} or~\cite[§6.1]{Rolla19}.

The proof of Theorem~\ref{thm:quantitative} follows a general framework introduced in~\cite{BasuGangulyHoffman18}. 
The key idea is to decompose space into blocks with independent instructions, so that they interact only through the number of walks arriving at the center of each block after having been ``emitted'' from a neighboring box.
This interaction is described by ``coarse-grained odometers,'' one for each block.
These are complicated random functions which are entangled by simple mass balance equations.
Instead of trying to say what the tuple of coarse-grained odometers are, one gets upper bounds by summing over all tuples compatible with the mass balance equations.
This approach ends up reducing the main bound~\eqref{eq:conditionu} to a single-block 
estimate (Proposition~\ref{prop:oneblock} below).
For the reader's convenience we reproduce this framework in §\ref{sec:framework}, following the description given in~\cite[§5.1]{Rolla19}.

The main contribution of this paper is to extend the single-block estimate of~\cite{BasuGangulyHoffman18}, which was proved for very small $\lambda$, to an estimate valid for all $\lambda < C^{-2} \, \zeta^2$.
This is done in §\ref{sec.single}.
In the proof we consider th single-block dynamics indexed by the jump times of the variable which counts the number of walks exiting from the left.
We show that the number of sleeping walks seen at these jump times has a drift downwards, and therefore has finite exponential moments.

Finally, in §\ref{sec:lower} we briefly show how the lower bound in Theorem~\ref{thm:main} follows from the proof of fixation given in~\cite{RollaSidoravicius12} and diffusive estimates for the $h$-transform of the simple random walk.

\section{Definitions and main tools}
\label{sec:definitions}

In this section we define more precisely the stochastic process to be studied, 
describe the site-wise representation, the Abelian property, and why~\eqref{eq:conditionu} implies the upper bound in Theorem~\ref{thm:main}.
We refer to the recent survey~\cite{Rolla19} for a more complete presentation.

\subsubsection*{The Abelian property}

A seemingly natural way to construct a collection of random walks is to sample sequences of instructions (going right, going left or sleeping) and attach them to the marks of each walk's clock.
But for a class of models which includes the ARW, it is convenient to attach the instructions to the sites of the graph instead.
In this setting, each walk is assigned a Poisson clock which determines \emph{when} the walk is going to perform an action, but \emph{the action itself} is determined by a stack of instructions \emph{assigned to the site} where the walk is.

These two ways of realizing the process are equivalent if the walks are seen as indistinguishable.
The latter construction provides a convenient coupling of the finite-volume dynamics on every $V \subseteq \Z$.
This coupling is very useful because of the celebrated \emph{Abelian property}.
Some aspects of the evolution, such as the final configuration and the number of visits to a site, are determined by the initial configuration and the stacks of instructions assigned to the sites, and do not depend on the Poisson clocks.

\subsubsection*{Formal definitions}

Let $\N = \{0,1,2,\dots\}$ and $\N_{\sleep} = \N \cup \{\sleep\}$, where $\sleep$ represents a sleeping walk.
For convenience we define $|\sleep|=1$, and $|n|= n$ for $n\in \N$, and write $0<\sleep<1<2<\cdots$.
Also define $\sleep+1=2$ and $n \cdot \sleep = n$ for $n \geq 2$ and $\sleep$ if $n=1$.

The state of the ARW at time $t\geqslant 0$ is given by $\eta_t \in (\N_{\sleep})^{\Z^d}$, and the process evolves as follows.
For each site $x$, a Poisson clock rings at rate $(1+\lambda) \, |\eta_t(x)| \, \I_{\eta_t(x) \ne \sleep}$.
When this clock rings, the system goes through the transition $\eta\to \toppling_{x\sleep}\eta$ with probability $\frac{\lambda}{1+\lambda}$, otherwise $\eta\to \toppling_{xy}\eta$ with probability $\frac{1}{2} \times \frac{1}{1+\lambda}$ for $y = x \pm 1$.
These transitions are given by
\[
 \toppling_{xy}\eta(z) =
 \begin{cases}
 \eta(x)-1, & z=x, \\
 \eta(y)+1, & z=y, \\
 \eta(z), & \mbox{otherwise,}
 \end{cases}
\qquad
 \toppling_{x\sleep}\eta(z) =
 \begin{cases}
 \eta(x) \cdot \sleep, & z=x, \\
 \eta(z), & \mbox{otherwise}
 \end{cases}
\]
and only occur if $\eta(x) \geq 1$.
The operator $\toppling_{x\sleep}$ represents a walk at $x$ trying to fall asleep, which effectively happen if there are no other walks present at $x$.
Otherwise, by definition of $n \cdot \sleep$,
the system state does not change.
The operator $\toppling_{xy}$ represents a walk jumping from $x$ to $y$, where possible activation of a sleeping walk previously found at $y$ is represented by the convention that $\sleep+1=2$.

Given a translation-invariant and ergodic distribution $\nu$ on $(\N_\sleep)^{\Z^d}$, let $\rho(\nu) = \int|\eta(\0)|\nu(\dd \eta)<\infty$ denote its average density.
If $\rho(\nu)<\infty$, there exists a process $(\eta_t)_{t\geqslant0}$ with transition rates described above and such that $\eta_0$ has law $\nu$.
We use $\bP^\nu$ to denote the underlying probability measure.
We say that the system \emph{fixates} if, for each $x\in \Z$, $\eta_t(x)$ remains constant for all $t$
large enough, otherwise we say that the system \emph{stays active}.

There exists a number $\zeta_c$, which is non-decreasing on $\lambda$, such that the system fixates a.s.\ if $\rho(\nu)<\zeta_c$ and stays active a.s.\ if $\rho(\nu)>\zeta_c$~\cite[Theorem~2.13]{Rolla19}.
Moreover, denoting $M_r$ the number of walks which exit $V_r$ in the finite-domain dynamics, if $\nu$ is a product measure and $\limsup_r \frac{\bE^{\nu}_{V_r} M_r}{r}>0$, then the system stays active a.s.~\cite[Theorem~2.11]{Rolla19}.
In particular, the bound~\eqref{eq:conditionu} implies the upper bound in Theorem~\ref{thm:main}.

We now describe the site-wise construction, which provides the Abelian property used in the proof of Theorem~\ref{thm:quantitative}.

\subsubsection*{Site-wise representation and stabilization}

We now use $\eta$ to denote configurations in $(\N_{\sleep})^{\Z^d}$ instead of a continuous-time process.
We say that site $x$ is \emph{unstable for the configuration $\eta$} if $\eta(x) \geqslant 1$. Otherwise, $x$ is said to be \emph{stable}.
By \emph{toppling} site $x$ we mean the application of an operator $\toppling_{xy}$ or $\toppling_{x\sleep}$ to $\eta$.
Toppling an unstable site is \emph{legal}.

Let $(\toppling^{x,j})_{x\in\Z^d, j\in\N}$ be a fixed field of instructions, that is, for each $x$ and $j$, $\toppling^{x,j}$ equals $\toppling_{x\sleep}$ or $\toppling_{xy}$ for some $y$.
Let $h\in \N^{\Z^d}$.
This field $h$ counts how many topplings occur at each site.
The toppling operation at $x$ is defined by
\(
\Phi_x(\eta,h)= \big(\toppling^{x,h(x)+1}\eta, h + \delta_x \big).
\)
Given a finite sequence $\boldsymbol{a}=(x_1,\dots,x_k)$, define
\(
\Phi_{\boldsymbol{a}} = \Phi_{x_k}\circ \Phi_{x_{k-1}}\circ \cdots \circ \Phi_{x_1}.
\)
We write $\Phi_{\boldsymbol{a}} \eta$ as a short for $\Phi_{\boldsymbol{a}} (\eta,0)$.
Given $V \subseteq \Z^d$, we say that $\eta$ is \emph{stable in $V$} if every $x\in V$ is stable for $\eta$.
We say that ${\boldsymbol{a}}$ is \emph{contained} in $V$ if $x_1,\dots,x_k \in V$.
We say that ${\boldsymbol{a}}$ \emph{stabilizes} $\eta$ in $V$ if $\Phi_{\boldsymbol{a}} \eta$ is stable in $V$.

The \emph{Abelian property} reads as follows.
If ${\boldsymbol{a}}$ and ${\boldsymbol{b}}$ are both legal toppling sequences for $\eta$ that are contained in $V$ and stabilize $\eta$ in $V$, then $\Phi_{\boldsymbol{a}} \eta = \Phi_{\boldsymbol{b}} \eta$.

We construct the measures $\Pb^{\eta_0}_V$ explicitly by taking all the $\toppling^{x,j}$ i.i.d.\ sampled with the distribution described above, plus Poisson clocks.
By the Abelian property, $S(V)$ it is determined by $\eta_0$ and $\Phi_{\boldsymbol{a}} \eta_0$, for \emph{any} sequence ${\boldsymbol{a}}$ of topplings which is contained in $V$ and stabilizes $\eta_0$ in $V$ (so the Poisson clocks will no longer be mentioned).
A convenient choice of ${\boldsymbol{a}}$, which is called a \emph{toppling procedure}, is absolutely central in §\ref{sec:framework} where we sketch the description of a machinery which relates the main result of §\ref{sec.single} to~\eqref{eq:expmoment}.

\section{Single-block estimate}
\label{sec.single}

In this section we state and prove a single-block estimate.
For $0 < \lambda \leq 1$, and $K \in 2\N$, define the domains $V = \{-K, \dots, K\}$ and $U=\{-\frac{K}{2},\dots,\frac{K}{2}\}$.
Let $\xi \in \{0,1\}^{U}$ be a fixed initial configuration supported on $U$.

\begin{figure}[b]
\centering
\includegraphics[width=.9\textwidth]{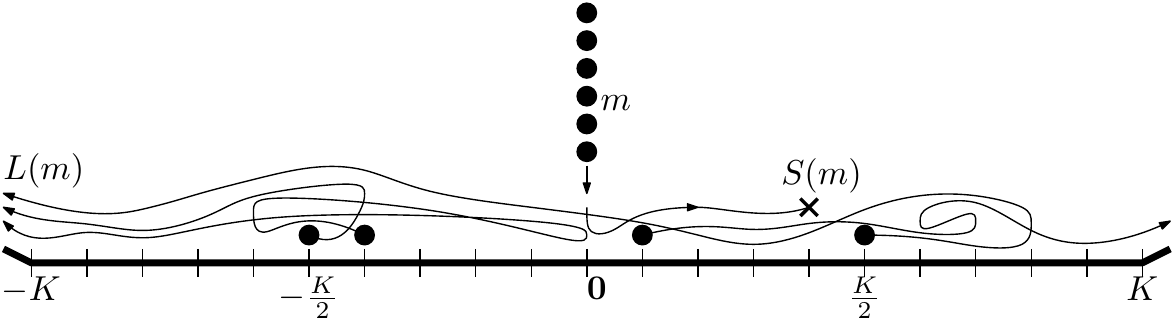}
\caption{Illustration of the dynamics inside a block with $K=10$}
\label{fig:oneblock}
\end{figure}

For $m=0,1,2,\dots$, consider stabilization of the configuration $\xi^m = \xi + m \delta_\0$ obtained by adding $m$ walks at site $x=\0$ to the configuration $\xi$.
Let $L(m)$ and $S(m)$ count how many walks exit $V$ from
the left and how many walks are sleeping in $V$ after $\xi^m$
is stabilized in $V$, see Figure~\ref{fig:oneblock}.
Dependence on $\xi$ and $V$ is omitted in the notation.

Note that $L(\cdot)$ and $S(\cdot)$ are non-decreasing random functions.

\begin{proposition}
\label{prop:oneblock}
For some $\alpha$, $\delta$ and $M$, for all $0<\lambda\leq \delta$,
taking $K = 2\lceil \frac{\delta}{\sqrt{\lambda}} \rceil$,
and $V = \{-K, \dots, K\}$
\begin{equation}
\label{eq:singleblock}
\sup_\xi
\sup_{\ell} \sum_{m \ge 0} \E_V^{\xi^m}
[e^{\alpha S(m)} \1_{L(m) = \ell} ] \le M
.
\end{equation}
\end{proposition}

By the Abelian property, for each fixed $m$, it does not matter whether we add the $m$ walks and stabilize $\xi^m$ at once, or whether we stabilize $\xi^0$, add a walk at $x=\0$, stabilize again, and repeat this process $m$ times.
It turns out that the latter is more convenient.

In words, the above estimate means the following.
We add walks at $x=\0$ and stabilize the resulting configuration.
We repeat this indefinitely, until a certain $m$ for which $L(m)=\ell$, that is, until exactly $\ell$ walks have exited form the left.
We then compute a factor of $e^{\alpha S}$, and a new one for each new walk addition, until another walk exits from the left.
After that, $L(m) \ge \ell+1$ and no more term contributes to the sum.

In the remainder of this section we prove the above estimate.
In the next section we show how it implies Theorem~\ref{thm:quantitative}.

We first reduce the problem to the case $\xi\equiv 0$,
by considering a sequence of positions $x \in U$ where new walks will be added, instead of placing them all at $x = \0$.
Let $k$ denote the number of walks in configuration $\xi$, and label the sites $x \in U$ with $\xi(x)=1$ as $x_{-k+1}, x_{-k+2}, \dots,x_{-1},x_0$.
Also, let $\0 = x_1 = x_2 = \cdots$, and write $\boldsymbol{x} = (x_m)_{m>-k}$.
For $m \ge -k$, define $\xi^m = \delta_{x_{-k+1}} + \delta_{x_{-k+2}} + \dots + \delta_{x_{m}}$.
In words, starting from $\xi^{-k} \equiv 0$, walks are added one by one until we arrive at $\xi^0=\xi$, and new walks are added at site $x=\0$ after that.
With this construction, we can define $L(m)$ and $S(m)$ for $m \ge -k$, starting from $L(-k)=S(-k)=0$.
Now note that we can bound the left-hand side of~\eqref{eq:singleblock} by a sum over $m \ge -k$.
By re-indexing the sum in~\eqref{eq:singleblock}, it is enough to show
\begin{equation}
\label{eq:xsingleblock}
\sup_{\boldsymbol{x}} \sup_{\ell} \sum_{m\ge 0}
 \E_V^{\xi^0} [ e^{\alpha S(m)} \1_{L(m) = \ell} ] \le M,
\end{equation}
with initial configuration $\xi^0 \equiv 0$ and the supremum over $\boldsymbol{x} \in U^\N$ instead.

We now proceed to the proof of~\eqref{eq:xsingleblock}, denoting $\E^{\xi^0}_V$ by $\E$.

Given $\epsilon>0$ to be fixed later, take $\delta$ such that a sleepy walk reaches distance $4\lceil \frac{\delta}{\sqrt{\lambda}} \rceil$ before falling asleep, with probability at least $1-\epsilon$ for every $0 <\lambda \leq \delta$.

So each time we release a walk in $V$, the probability of it falling asleep before exiting $V$ is at most $\epsilon$.
Also, each time we release a walk in $U$, the probability of exiting $V$ from the left before falling asleep is at least $p=\frac{1}{4}-\epsilon$.

The key idea of the proof is to look at times when $L$ increases:
\begin{equation}\label{def-tau}
\tau_\ell=\inf\{ \, m:\ L(m)\ge \ell \, \},\quad\text{for}\quad \ell\in \N.
\end{equation}
This allows us to read the sum over all values of
the odometer, in terms of the values of $S$ around $\tau_\ell$, through the simple identity
\[
\E\Big[\sum_{m\in \N} e^{\alpha S(m)}\1_{L(m)=l}\Big]=
\E\Big[\sum_{m=\tau_\ell}^{\tau_{\ell+1}-1} e^{\alpha S(m)}\Big].
\]
We will show that $(S(\tau_\ell))_{\ell = 0,1,2,\dots}$ is typically small as it has a drift downward.

Denote by $\F_\ell$ the $\sigma$-field of the instructions revealed when stabilizing $\xi^{\tau_\ell}$. 
If $L(\tau_\ell)>\ell$, then $\tau_{\ell+1}=\tau_\ell$ and of course $S(\tau_{\ell+1})=S(\tau_{\ell})$.

Suppose that $L(\tau_\ell)=\ell$ occurs.
Then $\tau_{\ell+1}\ge \tau_\ell+1$, and it is convenient to think of each walk we add after the $ \tau_\ell$-th as coming with its own trajectory independent of $\F_\ell$.
We mark each new walk as follows:
\begin{itemize}
\item A mark \emph{left} if it exits from the left before it tries to sleep.
\item A mark \emph{right} if it exits from the right before it tries to sleep.
\item A mark \emph{sleep} if it tries to sleep before it exits.
\end{itemize}
Let $G_{\ell+1}$ denote the number of walks added until the first marked \emph{left}.
Then
\[ \tau_{\ell+1} - \tau_\ell \le G_{\ell+1}. \]
Since each new walk is marked \emph{left} with probability at least $p$,
\[G_{\ell+1} \preccurlyeq G,\]
where $G$ denotes a geometric random variable with parameter $p$ and $\preccurlyeq$ denotes stochastic domination for the \emph{conditional distribution of $G_{\ell+1}$ given $\F_\ell$}.
Also, let $Z_{\ell+1}$ denote the number of walks marked \emph{sleep} before the $G_{\ell+1}$-th one.
Then $Z_{\ell+1}$ is a sum of $G_{\ell+1}-1$ independent Bernoulli variables of parameter at most $\epsilon$, hence
\[ Z_{\ell+1} \preccurlyeq Z, \]
where $Z+1$ is geometric with parameter $\frac{p}{p+\epsilon}$.

Assuming that not only $L(\tau_\ell)=\ell$ but also $S(\tau_\ell)>0$ occur, let $X_\ell$ denote the indicator of the event that the first walk addition after $\tau_\ell$ causes a sleeping walk to be reactivated and exit $V$.
Regardless of the position of the sleeping walks, the probability that the first walk added causes one of them to reactivate is at least $p$.
Once that occurs, the probability that the reactivated walk falls asleep again before exiting $V$ is at most $\epsilon$.
Hence, we have
\[ X_{\ell+1} \succcurlyeq X, \]
where $X$ is a Bernoulli with parameter $p - \epsilon$.

Thus, using notation $\nabla_\ell f=f(\tau_{\ell+1})-f(\tau_{\ell})$, when $L(\tau_\ell)=\ell$ we have
\begin{equation}\label{eq:change}
\nabla_\ell S \le Z_{\ell+1} - \1_{S(\tau_\ell) >0} \cdot X_{\ell+1}.
\end{equation}
Still on the event that $L(\tau_\ell)=\ell$, another key relation is
\begin{equation}
\label{eq:key-leftexit}
\nabla_\ell (L+S) \le 1+Z_{\ell+1}.
\end{equation}
Indeed, the left-hand side equals the number of walks added until $\tau_{\ell+1}$, minus the total number of walks which exit from the right.
This in turn equals the number $N$ of new walks added which are not marked \emph{right}, minus the number of sleeping walks that are reactivated and exit from the right.
Finally, $N$ is bounded by the right-hand side, proving~\eqref{eq:key-leftexit}.

To control the behavior of $S(\tau_\ell)$, we will bound it by the function
\begin{equation}
\nonumber
F(\ell)= 2 \times S(\tau_\ell) + L(\tau_\ell) - \ell.
\end{equation}

When $L(\tau_\ell)=\ell$, by adding relations
\eqref{eq:change} and \eqref{eq:key-leftexit}, we have
\begin{equation}
\label{eq:drift}
F(\ell+1)-F(\ell) \le 2 Z_{\ell+1}- \1_{S(\tau_\ell) >0} \cdot X_{\ell+1}.
\end{equation}

When $L(\tau_\ell)>\ell$, we have $F(\ell+1)-F(\ell)=-1$.
To keep the computations short, we still use~\eqref{eq:drift} in this case.
Note that indeed this estimate is still valid if we set $Z_{\ell+1}=0 \preccurlyeq G$ and $X_{\ell+1}=1 \succcurlyeq X$ on this event.

Denoting $\tilde\E = \E[ \,\cdot\, | \F_\ell]$, from~\eqref{eq:drift} we get
\begin{equation}
\label{eq:always}
\tilde\E e^{\alpha F({\ell+1})}
\le
(\tilde\E e^{2 \alpha Z_{\ell+1}}) \cdot \1_{F(\ell) = 0}
+
e^{\alpha F(\ell)} (\tilde\E e^{\alpha (2Z_{\ell+1}-X_{\ell+1})}) \cdot \1_{F(\ell) > 0}
.
\end{equation}
Using Cauchy-Schwarz,
\[
(\tilde\E e^{\alpha (2Z_{\ell+1}-X_{\ell+1})})^2
\leq
\tilde\E e^{4 \alpha Z_{\ell+1}}
\times
\tilde\E e^{- 2 \alpha X_{\ell+1}}
\leq
\E e^{4 \alpha Z}
\times
\E e^{- 2 \alpha X}
=:
\beta^2.
\]
Taking expectation in~\eqref{eq:always} we get
\begin{equation}
\nonumber
\E e^{\alpha F({\ell+1})}
\le
\E e^{2 \alpha Z} + \beta \, \E e^{\alpha F(\ell)}
.
\end{equation}

Now choose $\varepsilon$ small so that $\E X > 2\E Z$, and $\alpha$ small so that $0<\beta<1$.
Iterating the previous inequality gives
\[
\E e^{\alpha F(\ell+1)}
\le
(1+\cdots+\beta^\ell)\E e^{2 \alpha Z} + \beta^\ell \, \E e^{\alpha F(0)}
.
\]
To get an estimate not depending on $\ell$, we make the crude bound
\begin{equation}
\label{eq:uniformwithtau}
\E e^{\alpha F({\ell})}
\le 1 + \tfrac{1}{1-\beta}\E e^{2 \alpha Z}
< \infty
.
\end{equation}

To conclude, we make another crude estimate
\[
\max\{ S(\tau_\ell),S(\tau_\ell+1),\dots,S(\tau_{\ell+1}) \}
\leq
S(\tau_\ell) + (\tau_{\ell+1} - \tau_\ell)
\]
Combined with~\eqref{eq:uniformwithtau}, this gives for any fixed $\ell$
\begin{align*}
\E \Big[ \sum_m e^{\alpha S(m)} \1_{L(m) = \ell} \Big]
& =
\E \Big[ \sum_{m=\tau_\ell}^{\tau_{\ell+1}-1} e^{\alpha S(m)} \Big]
\\
& \le
\E \Big[ \E \Big[
e^{\alpha [S(\tau_\ell) + (\tau_{\ell+1} - \tau_\ell)]} (\tau_{\ell+1} - \tau_\ell) \Big| \F_\ell
\Big] \Big]
\\
& \le
\E e^{\alpha S(\tau_\ell)}
\times
\E [ G e^{\alpha G}]
\\
& \le
(1 +
\tfrac{1}{1-\beta}
\E e^{2 \alpha Z})
\times
\E [ G e^{\alpha G}]
=: M.
\end{align*}
By further reducing $\alpha$, we make $\E [ G e^{\alpha G}]<\infty$.
This establishes~\eqref{eq:xsingleblock} and concludes the proof of Proposition~\ref{prop:oneblock}.

\section{Proof of exponential moment}
\label{sec:framework}

In this section we show how Theorem~\ref{thm:quantitative} follows from Proposition~\ref{prop:oneblock} and a general framework introduced in~\cite{BasuGangulyHoffman18}.
We sketch the main features of the construction, referring the reader to~\cite[§5.1]{Rolla19} for the details.

\subsubsection*{The toppling procedure}

Let $K \in 2\N$ be given.
Later on it will be chosen as in Proposition~\ref{prop:oneblock}.
We can suppose $2r+1=(K+1)n$ for some positive integer $n$.
We can also suppose that $\eta_0 \in \{0,1\}^{V_r}$, otherwise we simply topple every site containing two or more walks until there is no longer such a site,
and start from the resulting configuration.

We assign a different \emph{color} to each \emph{source}, that is a lattice site whose position is $s_i=i(K+1)-r$, for $i=1,\dots,n$.
Sites are grouped into \emph{blocks} numbered $i=1,\dots,n$, of the form $\{s_i-K,\dots,s_i+K\}$,
centered around a \emph{source}.
Each site which is not a source belongs to two blocks, and is assigned two independent stacks of instructions, one for each block.
Walks get the color of the last source they visited (initially they are assigned the color of the nearest source), and use only instructions of their own color.
We only let walks reactivate other walks if they have the same color, if they have different colors we treat the sleeping walk as if it was still sleeping.
The Abelian property implies that the configuration at the end of this procedure (which imposes a restriction on re-activation) gives an upper bound for $S(V_r)$.

\subsubsection*{Single-block dynamics}

For $m\in\N$, consider the stabilization of the configuration $\xi + m \delta_{s_i}$ inside the $i$-th block.
That is, $m$ walks are added to the source $s_i$ and the configuration is toppled until it is stable in the block.

We now define random functions denoted by
\(
L_i(\cdot)
,
\
R_i(\cdot)
\
\text{and}
\
S_i(\cdot)
,
\)
illustrated in Figure~\ref{fig:oneblock}.
Let $L_i(m)$ count the number of walks that exit the block from the left when $\xi + m \delta_{s_i}$ is stabilized in the $i$-th block, let $R_i(m)$ count the number of walks that exit the block from the right, $S_i(m)$ the number of walks sleeping in the block.

\subsubsection*{Mass balance equations and proof of active phase}

Let $m_i^*$ denote the number of times a walk arrives at the $i$-th source and acquires its color.
Writing
$\tm^* = (m_1^*,\dots,m_n^*)$,
$R_0 \equiv 0$ and $L_{n+1} \equiv 0$, the vector $\tm^*$ satisfies the \emph{mass balance equations}
\begin{equation}
\label{eq:buffer}
m_i = R_{i-1}(m_{i-1}) + L_{i+1}(m_{i+1})
\qquad
\text{ for }
i=1,\dots,n.
\end{equation}
We now rewrite the above system as
\begin{equation}
\label{eq:realizable}
L_i(m_i) = m_{i-1} - R_{i-2}(m_{i-2})
\end{equation}
for $i=1,\dots,n+1,$
where $R_{-1}\equiv 0$ and $m_{0}$ can be taken as $L_1(m_1)$.

\subsubsection*{Estimating the exponential moment}

Choose $\alpha$, $\delta$ and $M$ according to Proposition~\ref{prop:oneblock}.
Given $0 < \lambda \le \delta$, take $K = 2\lceil \frac{\delta}{\sqrt{\lambda}} \rceil$.
Also recall that $r=(n+1)K$.

For a non-negative vector $\tm$, define
\[
S(\tm) = \sum_{i=1}^n S_i(m_i)
.
\]
The total number of walks present in the blocks after global stabilization is given by
\[
S^* = S(\tm^*).
\]
Recalling that~\eqref{eq:realizable} is satisfied for $i=1,\dots,n$ when $\tm=\tm^*$, we have
\begin{align*}
\E e^{\alpha S^*} &=
\sum_{\tm} \E[ e^{\alpha S(\tm)} \I_{\tm^* = \tm}]
\\
& \le \sum_{m_0}
\E \bigg[ \sum_{m_1} \dots \sum_{m_{n}}
\prod_{i=1}^n e^{\alpha S_i(m_i)} \I_{\{L_i(m_i) = m_{i-1} - R_{i-2}(m_{i-2})\}} \bigg] .
\end{align*}
Then after taking successively conditional expectations with respect to the filtrations generated by $(L_j(\cdot),R_j(\cdot),S_j(\cdot))_{j =1,\dots k}$, for $k=n-1,\dots,1$,
one concludes that (see~\cite[§5.1]{Rolla19})
\begin{align*}
\E e^{\alpha S^*} \le
\sum_{m_0} \prod_{i=1}^n \sup_\ell \E \bigg[ \sum_{m_i}
e^{\alpha S_i(m_i)} \I_{\{L_i(m_i) = \ell\}} \bigg]
\le r M^{r/K},
\end{align*}
using Proposition~\ref{prop:oneblock} for the last inequality.
Since $S(V_r)$ is stochastically dominated by $S^*$, this concludes the proof of
Theorem~\ref{thm:quantitative}.

\section{Diffusive lower bound}
\label{sec:lower}

The lower bound in Theorem~\ref{thm:main} follows from two facts.

Let $(X_n)_{n \geq 0}$ be a simple symmetric random walk starting from $X_0=0$ and
conditioned to be positive for all $n>0$, that is, $X$ is the $h$-transform of a random walk.
Let $Z_n = \max\{X_1,\dots,X_n\}$.

\begin{lemma}[{\cite[Remark~4.3]{Rolla19}}]
Let $\NN$ be a geometric random variable
with parameter $\frac{\lambda}{1+\lambda}$ independent of $X$.
Then
\[
\zeta_c(\lambda) \ge \frac{1}{\E Z_{\NN}}.
\]
\end{lemma}

The lower bound in~\eqref{main-estimate} is a consequence of this and the following.

\begin{lemma}
\label{lem-hunt}

There exists a constant $C$ such that, for every $n$,
\[
\E Z_n \le C \sqrt{n}.
\]
\end{lemma}
\begin{proof}
We learned this elegant proof from Nicolas Curien.
First we claim that $\E X_n = \frac{1}{2n} \E |S_n^3|$, with $(S_n)_{n\ge 0}$ the unconditioned walk.
Indeed, recall that $X$ is an $h$-process, with $h(x) = x$. Therefore for any $x>0$,
\[
\Pb\big[ X_n=x \big] = x \, \Pb \big[S_n = x, \, S_k>0 \ \forall k=1,\dots,n \big] = \frac{x^2}{n} \Pb \big[ S_n=x \big],
\]
using the Cyclic Lemma for the second equality.
The claim follows if we multiply by $x$ and sum over $x>0$.
Second, $\E Z_n \leq 1 + 4 \E X_n$.
To prove the latter, observe that $[h(X_n)]^{-1}$ is a martingale.
Hence,
\begin{multline*}
\Pb \big[ X_n \geq k \big] \geq \Pb \big[ \tau_{2k} \leq n \text{ and } X_j > k \text{ for all } j>\tau_{2k} \big]
= \\ =
\Pb \big[ \tau_{2k} \leq n \big] \times \Pb^{2k} [ \tau_k = \infty ] = \tfrac{1}{2} \Pb \big[ Z_n \geq 2k \big],
\end{multline*}
and summing over $k$ gives the inequality.
To conclude just observe that $\E|S_n^3| \leq 3 n^{3/2}$ by computing $\E S_n^4 = 3n(n-1)\E S_1^2 + n\E S_1^4\le 3n^2$ and using Jensen's inequality.
\end{proof}

\section*{Acknowledgments}

We thank Nicolas Curien for valuable help with the $h$-process, and Perla Sousi for enlightening discussions at an early stage of this project.

\bibliographystyle{bib/leo}
\bibliography{bib/leo}


Amine Asselah
\\
Université Paris-Est, LAMA (UMR 8050), UPEC, UPEMLV, CNRS, F-94010, Cr\'eteil, France
\\
amine.asselah@u-pec.fr

Leonardo T. Rolla
\\
University of Buenos Aires and NYU-Shanghai
\\
leorolla@dm.uba.ar

Bruno Schapira
\\
Aix-Marseille Université, CNRS, Centrale Marseille, I2M, UMR 7373, 13453 Marseille, France
\\
bruno.schapira@univ-amu.fr

\end{document}